\documentclass{amsart}
\usepackage{amsfonts,amssymb,amscd,amsmath,enumerate,verbatim,calc}
\usepackage[all]{xy}

\newcommand{\CM}{Cohen-Macaulay}

\newcommand{\wrt}{with respect to}

\newcommand{\n}{\mathfrak{n} }
\newcommand{\m}{\mathfrak{m} }

\newcommand{\q}{\mathfrak{q} }

\newcommand{\ZZ}{\mathbb{Z} }
\newcommand{\FF}{\mathbb{F}}
\newcommand{\GG}{\mathbb{G}}
\newcommand{\HH}{\mathbb{H}}

\newcommand{\bx}{\mathbf{x} }

\newcommand{\rt}{\rightarrow}

\newcommand{\ov}{\overline}

\newcommand{\Om}{\Omega }
\newcommand{\image}{\operatorname{image}}
\newcommand{\coker}{\operatorname{coker}}
\newcommand{\Ass}{\operatorname{Ass}}

\newcommand{\depth}{\operatorname{depth}}
\newcommand{\Tr}{\operatorname{Tr}}

\newcommand{\ann}{\operatorname{ann}}

\newcommand{\codim}{\operatorname{codim}}

\newcommand{\projdim}{\operatorname{projdim}}

\newcommand{\Hom}{\operatorname{Hom}}
\newcommand{\Ext}{\operatorname{Ext}}
\newcommand{\Tor}{\operatorname{Tor}}
\newcommand{\CMS}{\operatorname{\underline{CM}}}
\newcommand{\ICMS}{\operatorname{\underline{ICM}}}
\newcommand{\CMa}{\operatorname{CM}}
\newcommand{\CMr}{\operatorname{CM^r}}
\newcommand{\ICMa}{\operatorname{ICM}}

\theoremstyle{plain}

\newtheorem{theorem}{Theorem}[section]
\newtheorem{corollary}[theorem]{Corollary}
\newtheorem{lemma}[theorem]{Lemma}
\newtheorem{proposition}[theorem]{Proposition}

\theoremstyle{definition}
\newtheorem{definition}[theorem]{Definition}
\newtheorem{s}[theorem]{}
\newtheorem{remark}[theorem]{Remark}

\theoremstyle{remark}

\begin{document}

\title[  stable category ]{A function on the the set of isomorphism classes in the stable category of maximal Cohen-Macaulay modules over a Gorenstein ring: with applications to Liason theory }
 \author{Tony J. Puthenpurakal}
\date{\today}
\address{Department of Mathematics, Indian Institute of Technology Bombay, Powai, Mumbai 400 076, India}
\email{tputhen@math.iitb.ac.in}


\begin{abstract}
Let $(A,\m)$ be a Gorenstein local ring of dimension $d \geq 1$. Let $\CMS(A)$ be the stable category of maximal \CM \ $A$-modules and let $\ICMS(A)$ denote the set of isomorphism classes  in $\CMS(A)$. We define a function $\xi \colon \ICMS(A) \rt \ZZ$ which behaves well with respect to exact triangles in $\CMS(A)$. We then apply this to (Gorenstein) liason theory. We prove that if $\dim A \geq 2$ and $A$ is not regular then the even liason classes of $\m^n; n\geq 1$  is an infinite set. We also prove that if $A$ is an complete equi-characteristic  simple singularity  with $A/\m$ uncountable then for each $m \geq 1$ the set $\mathcal{C}_m = \{ I \mid  I \ \text{is a codim $2$ CM-ideal with} \ e_0(A/I) \leq m \}$ is contained in finitely many even liason classes $L_1,\ldots,L_r$ (here $r$ may depend on $m$).

\end{abstract}

\maketitle
\section{Introduction}
Let $(A,\m)$ be a Gorenstein local ring of dimension $d \geq 1$ with residue field $k$. Let $\CMa(A)$ denote the full subcategory of maximal \CM \ $A$-modules and let $\CMS(A)$ denote the stable category of maximal \CM \ $A$-modules. Recall that objects in $\CMS(A)$ are same as objects in $\CMa(A)$. However the set of morphisms $\underline{\Hom_A}(M,N)$ between $M$ and $N$ is  $= \Hom_A(M,N)/P(M,N)$ where $P(M,N)$ is the set of $A$-linear maps from $M$ to $N$ which factor through a finitely generated free module. It is well-known that $\CMS(A)$ is  a triangulated category with translation functor $\Omega^{-1}$. Here $\Omega(M)$ denotes the syzygy of $M$ and $\Omega^{-1}(M)$ denotes the co-syzygy of $M$. Also recall that an object $M$ is zero in $\CMS(A)$ if and only if it is free  considered as an $A$-module. Furthermore $M \cong N$ in $\CMS(A)$ if and only
if there exists finitely generated free modules $F,G$ with $M\oplus F \cong N \oplus G$ as $A$-modules. 
Let $\ICMS(A)$ denote the set of isomorphism classes  in $\CMS(A)$ and for an object $M \in \CMS(A)$ we denote its isomorphism class by $[M]$.

We say a function $\xi \colon \ICMS(A) \rt \ZZ$ is a \emph{triangle} 
 function if it satisfies the following properties:
\begin{enumerate}
\item
$\xi([M]) \geq 0$ for all $M \in \CMS(A)$.
\item
$\xi([M]) = 0$ if and only if $M = 0$ in $\CMS(A)$ .
\item
$\xi([M_1 \oplus M_2]) = \xi([M_1]) + \xi([M_2])$ for all $M_1, M_2 \in \CMS(A)$.
\item 
\emph{
(sub-additivity)} If $M \rt N \rt L \rt \Om^{-1}(M)$ is an exact triangle in $\CMS(A)$ then
\begin{enumerate}
\item
$\xi([N]) \leq \xi([M]) + \xi([L])$.
\item
$\xi([L]) \leq \xi([N]) + \xi([\Om^{-1}(M)])$.
\item
$\xi([\Om^{-1}(M)]) \leq \xi([L]) + \xi([\Om^{-1}(N)])$.
\end{enumerate}
\end{enumerate}
\begin{remark}\label{rotation}
\begin{enumerate}[\rm (i)]
\item
Since rotations of exact triangles are exact it follows that if $\xi$ satisfies (4)(b) for all exact triangles then it will also satisfy 4(a),(c).
\item
Axiom (3) implies that $\xi([M]) = 0$ if $M = 0$. However note that axiom (2) also implies that if $\xi([M]) = 0$ then $M = 0$.
\end{enumerate}
\end{remark}
We have the following result on existence of triangle functions. Let $\ell(N)$ denote the length of an $A$-module $N$.
\begin{theorem}\label{existence}
Let $(A,\m)$ be a Gorenstein local ring of dimension $d \geq 1$. Then the function
$$ e^T_A([M]) = \lim_{n \rt \infty} \frac{(d-1)!}{n^{d-1}}\ell\left(\Tor^A_1(M, \frac{A}{\m^{n+1}}) \right )  \quad \text{where} \  [M] \in \ICMS(A)$$
is a triangle function on $\ICMS(A)$.
\end{theorem}

Unlike the multiplicity function which can be defined uniquely through a set of axioms, triangle functions are highly non-unique. In \ref{inf-tri} we will construct infinitely many triangle functions. However $e^T_A$ is the simplest triangle function that we have constructed. It also behaves well with generic hyperplane sections, see Proposition \ref{mod-sup} for details.

\subsection{Applications to Liason theory } \  \\
The existence of triangle functions has non-trivial implications in Liason theory. In  fact in Application I and II we prove our results by using any triangle function. However for application III we need some additional properties of $e^T_A$. 
 
Let $(A,\m)$ be a Gorenstein local ring.
We say an ideal $\q$ is a Gorenstein ideal if it is perfect  and $A/\q$ is a Gorenstein ring.
We should remark that some authors \textit{do not} require in the definition of Gorenstein ideals for $\q$ to be perfect. \textit{However we will require it to be so.}

We begin by recalling the definition of (Gorenstein)  linkage. 
\begin{definition}
Ideals $I$ and $J$ of $A$ are (algebraically) linked by a Gorenstein ideal $\q$ if
\begin{enumerate}[\rm (a)]
\item
$\q \subseteq I\cap J$, and
\item
$I = (\q \colon J)$ and $J = (\q \colon I)$.
\end{enumerate}
We write it as $I \sim_\q J$.
\end{definition}
If $\q$ is a complete intersection ideal then we say that $I$ is CI-linked to $J$.
We say ideals $I$ and $J$ is in the \emph{same linkage class} if there is a sequence of ideals $I_0,\ldots, I_n$ in $A$ and Gorenstein ideals $\q_0 \ldots,\q_{n-1}$ such that
\begin{enumerate}[\rm (i)]
\item
$I_j \sim_{\q_j} I_{j+1}$, for $j = 0,\ldots, n-1$.
\item
 $I_0 = I$ and $I_n =J$. 
\end{enumerate}
If $n$ is even then we say that $I$ and $J$ are \emph{evenly linked}.
We can analogously define CI-linkage class and even CI-linkage class.

The notion of linkage has been extended to modules, \cite{MS}. See section $4$ for definition. Note that ideals $I$ and $J$ are linked as ideals if and only if the cyclic modules
$A/I$ and $A/J$ are linked as modules; see \cite[Proposition 1]{MS}

\noindent\emph{Application- I:}  Let $K$ be a field  let $R = K[[X_1,\ldots,X_n]]$. Set  $\n = (X_1,\ldots, X_n)$.  By results in \cite[Theorem 3.6]{KM} it can be shown that $\n^i$ is evenly linked to $\n^{i-1}$ for all $i \geq 2$.  Note that if $n \geq 3$ then this result \textit{does not} hold for CI-liason \cite[Theorem 1.1]{W}.
If $(A,\m)$ is a one dimensional Gorenstein local ring  then one can prove that there exists $s \geq 1$ such that $\m^{sn + r}$ is evenly linked to $\m^{s(n-1) + r}$ for all $n \gg 0$ and $r = 0,1,\ldots,s-1$; see Proposition \ref{1-dim}.
A natural question is when is the set of ideals $\{ \m^n \mid n \geq 1 \} $ contained in finitely many even liason classes. 
Our first result implies that the above two cases 
are essentially the only cases when the above condition holds. We prove the following more general result:
\begin{theorem}\label{result-1}
Let $(A,\m)$ be a Gorenstein local ring. Let $M$ be a finitely generated $A$-module of dimension $r \geq 2$.
Let $\Lambda_M = \{ M/\m^n M  \mid n \geq 1 \}$. If there exists finitely many even liason classes of modules $L_1, L_2 ,\ldots, L_m$ such that
\[
\Lambda_M \subseteq \bigcup_{i = 1}^{m} L_i
\]
then $A$ is regular.
\end{theorem}

\noindent\emph{Application-II:}  Assume $(A,\m)$ is a complete equi-characteristic Gorenstein local ring. Let $I$ be an ideal in $A$ generated by a regular sequence. Using results in \cite[Theorem 3.6]{KM} it can be proved that $I^n$ is evenly linked to $I^{n-1}$ for all $n \geq 2$, see Proposition \ref{reg-equi}. Thus the modules
$A/I^n$ is evenly linked to $A/I^{n-1}$ for all $n \geq 2$. It follows that if $F$ is a finitely generated free $A$-module then $F/I^nF$ is evenly linked to $F/I^{n-1}F$ for all $n \geq 2$.
A natural question is whether $M/I^nM$ is  evenly linked to $M/I^{n-1}M$
for all $n \gg 0$ when $M \in \CMa(A)$.
We prove the following surprising result:
\begin{theorem}\label{result-2}
Let $(A,\m)$ be a Gorenstein local ring of dimension $d \geq 2$. Let $M \in \CMa(A)$. Let $x_1,\ldots,x_r$ be an $A$-regular sequence with $r \geq 2$. Let $I = (x_1,\ldots,x_r)$  and let
$\Lambda_M^I = \{ M/I^nM \mid n \geq 1 \}$. If there exists finitely many even liason classes of modules $L_1, L_2 ,\ldots, L_m$ such that
\[
\Lambda_M^I \subseteq \bigcup_{i = 1}^{m} L_i
\]
then $M$ is free.
\end{theorem}
Note that in the above result we \emph{do not} assume that $A$ is complete or contains a field. We do not know whether the result holds if $r =1$.

\noindent\emph{Application-III:}  Let $I$ be a perfect ideal of codimension 2. It is well-known that $I$ is licci (i.e., it is CI-linked to a complete intersection). However an arbitrary  codimension two Cohen-Macaulay ideal need not be licci. For instance if $(A,\m)$ is non-regular Gorenstein ring of dimension 2 then $\m$ is not a licci-ideal (this is so because if $I$ is licci then $\projdim A/I $ is finite.) So a natural question is whether codimension two Cohen-Macaulay ideals are contained in finitely many even liason classes. Again this is not possible. Let $(A,\m)$ be a  non-regular Gorenstein ring of dimension 2. Then by Theorem \ref{result-1}  the set of ideals $\{ \m^n \mid n\geq 1 \}$ is not contained in finitely many even liason classes of ideals in $A$. Note that $\ell(A/\m^n) \rightarrow \infty$ as $n \rightarrow \infty$. So we reformulate the question. Let $\mathcal{C}_m = \{ I \mid  I \ \text{is a codim $2$ CM-ideal with} \ e_0(A/I) \leq m \}$. Here $e_0(A/I)$ is the multiplicity of the ring $A/I$ \wrt \ its maximal ideal. Our question is whether $\mathcal{C}_m$ contained in finitely many even liason classes of ideals. Regular rings trivially have this property.
We prove 
\begin{theorem}\label{result-3}
Let $k$ be an uncountable algebraically closed field of characteristic different from $2$. Let $(P,\n)$ be a regular analytic $k$-algebra, i.e., a formal or convergent (if $k$ is a complete non-trivial valuated field) power series ring $k<x_0,x_1,\ldots,x_d>$ with $d \geq 2$.  Let $f \in \n$ be such that $A = P/(f)$ is a simple hypersurface singularity. For $m \geq 1$ let $\mathcal{C}_m = \{ I \mid  I \ \text{is a codim $2$ CM-ideal with} \ e_0(A/I) \leq m \}$. Then for every $m \geq 1$ there exists finitely many even liason classes $L_1,\ldots,L_r$ (depending on $m$) such that
\[
\mathcal{C}_m \subseteq \bigcup_{i = 1}^{r} L_i
\]
\end{theorem}
In the above result we use the fact that a simple simple singularity only has finitely many indecomposable maximal \CM \ modules, see \cite{K}. The assumption $K$ is uncountable is a bit irritating, however it is essential in our proof. We conjecture that the converse of this theorem is also true. 

We now describe in brief the contents of the paper. In section two we introduce the function $e^T_A(-)$ and prove some of its basic properties. In section three we prove Theorem \ref{existence}. In section 4 we discuss some results on Liason theory of modules and discuss the notion of maximal \CM \ approximations. In section 5,6,7 we prove Theorems \ref{result-1}, \ref{result-2},\ref{result-3} respectively. 
\section{Pre-triangles in $\CMa(A)$}
In this paper all rings are commutative Noetherian and all modules are assumed to be finitely generated.
In this section $(A,\m)$ is a Cohen-Macaulay local ring of dimension $d \geq 1$. Let $\ICMa(A)$ denote the set of isomorphism classes of maximal \CM \ $A$-modules and for an object $M \in \CMa(A)$ we denote its isomorphism class by $[M]$. In this section we study the function 
$$ e^T_A([M]) = \lim_{n \rt \infty} \frac{(d-1)!}{n^{d-1}}\ell\left(\Tor^A_1(M, \frac{A}{\m^{n+1}}) \right )  \quad \text{where} \  [M] \in \ICMa(A).$$ 
We also abstract some of its properties and call the notion a \emph{pre-triangle} function.
\s Let $M$ be an $A$-module. We denote it's first syzygy-module by $\Om(M)$. If we have to specify the ring then we write it as $\Om_A(M)$. Recall
$\Om(M)$ is constructed  as follows: Let $G \xrightarrow{\phi} F \xrightarrow{\epsilon} M \rt 0$ be a minimal presentation of $M$. Then $\Om(M) = \ker \epsilon$. It is easily shown that  if $G^\prime \xrightarrow{\phi^\prime} F \xrightarrow{\epsilon^\prime} M \rt 0$ is another  minimal presentation of $M$ then
$\ker \epsilon \cong \ker \epsilon^\prime$. 

Set $\Om^1(M) = \Om(M)$. For $i \geq 2$  define $\Om^i(M) = \Om(\Om^{i-1}(M))$. It can be easily proved that $\Om^i(M)$ are invariant's of $M$.

\s The function $e^T_A(-)$ arose in the authors study of certain aspects of the  theory of Hilbert functions \cite{Pu1},\cite{Pu2}. 
Let $N$ be a $A$-module of dimension $r$. It is well-known that there exists a polynomial $P_N(z) \in \mathbb{Q}[z]$ of degree $r$ such that $P_N(n) = \ell(N/\m^{n+1}N)$ for all $n \gg 0$.
We write
\[
P_N(z) = \sum_{i = 0}^{r}(-1)^ie_i(N)\binom{z+r-i}{r-i}.
\]
Then $e_0(N),\cdots,e_r(N)$ are integers and are called the \textit{Hilbert coefficients} of $N$. 
The number $e_0(N)$ is called the \textit{multiplicity} of $N$. It is positive if $N$ is non-zero. The number $e_1(N)$ is \textbf{non-negative} if $N$ is \CM; see \cite[Proposition 12]{Pu1}.  Also note that
\[
\sum_{n\geq 0}\ell(N/\m^{n+1}N) z^n  = \frac{h_N(z)}{(1-z)^{r+1}};
\]
where $h_M(z)\in \ZZ[z]$ with $e_i(N) = h^{(i)}(1)/i!$ for $i = 0,\ldots,r$.

\s\label{basic} Let $M \in \CMa(A)$. In \cite[Prop. 17]{Pu1} we proved that the function
$$ n \mapsto \ell\left(\Tor^A_1(M, \frac{A}{\m^{n+1}}) \right )$$
is of polynomial type, i.e., it coincides with a polynomial $t_M(z)$ for all $n \gg 0$. In \cite[Theorem 18]{Pu1} we also proved that
\begin{enumerate}
\item
$M$ is free if and only if $\deg t_M(z) < d-1$.
\item
If $M$ is not free then $\deg t_M(z) = d-1$ and the normalized leading coefficient of $t_M(z)$ is 
$\mu(M)e_1(A) - e_1(M) - e_1(\Om(M))$; here $\mu(M)$ denotes the minimal number of generators of $M$.
\item
For any $M \in \CMa(A)$,
\begin{align*}
e^T_A(M) &= \lim_{n \rt \infty} \frac{(d-1)!}{n^{d-1}}\ell\left(\Tor^A_1(M, \frac{A}{\m^{n+1}}) \right ) \\
          &= \mu(M)e_1(A) - e_1(M) - e_1(\Om(M)).
\end{align*}
\end{enumerate}
 By (1) note that $e^T_A(M) = 0$ if and only if $M$ is free. Otherwise $e^T_A(M) > 0$ Infact $e^T_A(M) \geq e_0(\Om(M))$; \cite[Lemma 19]{Pu1}.
 
 Our first result shows that we need not confine to minimal presentation to compute
 $e_A^T(M)$. 
\begin{lemma}\label{arbit-pres}
Let $M \in \CMa(A)$ and let $0 \rt N \rt F \rt M \rt 0$ be an exact sequence in $\CMa(A)$ with $F$ free. Then
\[
e^T_A(M) = e_1(F)-e_1(M) -e_1(N).
\]
\end{lemma}
\begin{proof}
By Schanuel's Lemma \cite[Lemma 3, section 19]{Mat} we have $A^{\mu(M)}\oplus N \cong F \oplus \Om(M)$. So 
\[
\mu(M)e_1(A) + e_1(N) = e_1(F) + e_1(\Om(M)).
\]
The result follows.
\end{proof}

Our next result shows that $e_1(-)$ is sub-additive over short-exact sequences in $\CMa(A)$.

\begin{proposition}
Let $0 \rt M_1 \rt M_2 \rt M_3 \rt 0$ be a short-exact sequence in $\CMa(A)$. Then
$$ e_1(M_2) \geq e_1(M_1) + e_1(M_3).$$ 
\end{proposition}
\begin{proof}
Note $e_0(M_2) = e_0(M_1) + e_0(M_3)$. For $n \geq 0$ we define modules $K_n$ by the exact sequence
\[
0 \rt K_n \rt \frac{M_1}{\m^{n+1}M_1} \rt \frac{M_2}{\m^{n+1}M_2} \rt \frac{M_3}{\m^{n+1}M_3} \rt 0.
\]
It follows that
\[
\sum_{n\geq 0} \ell(K_n)z^n  = \frac{h_{M_1}(z) - h_{M_2}(z) + h_{M_3}(z)}{(1-z)^{d+1}}.
\]
Since $e_0(M_2) = e_0(M_1) + e_0(M_3)$ we have that $ h_{M_1}(z) - h_{M_2}(z) + h_{M_3}(z)
= (1-z)l_K(z)$ for some $l_K(z) \in \ZZ[z]$. So we have
\[
\sum_{n\geq 0} \ell(K_n)z^n  = \frac{l_K(z)}{(1-z)^d}.
\]
Notice $l_K(1) = e_1(M_2) - e_1(M_1) - e_1(M_3)$. It follows that for all $n \gg 0$
\[
 \ell(K_n) = \left( e_1(M_2) - e_1(M_1) - e_1(M_3)\right)\frac{n^{d-1}}{(d-1)!} + \ \text{lower terms in } \ n.
\]
So $e_1(M_2) \geq  e_1(M_1) + e_1(M_3)$.
\end{proof}
We now prove that $e^T_A(-)$ is sub-additive over short-exact sequences in $\CMa(A)$.
\begin{theorem}\label{sub-add}
Let  $0 \rt M_1 \rt M_2 \rt M_3 \rt 0$ be a short-exact sequence in $\CMa(A)$. Then
\[
e^T_A(M_2) \leq e^T_A(M_1) + e^T_A(M_3).
\]
\end{theorem}
\begin{proof}
By a standard result in homological algebra we have the following diagram with exact rows and columns; with $F_i$ free $A$-modules for $i = 1,2,3$:
\[
  \xymatrix
{
\
&0
\ar@{->}[d]
&0
\ar@{->}[d]
&0
\ar@{->}[d]
\
\\
 0
 \ar@{->}[r]
  & N_1
\ar@{->}[r]
\ar@{->}[d]
 & N_2
\ar@{->}[r]
\ar@{->}[d]
& N_3
\ar@{->}[r]
\ar@{->}[d]
&0
\\
 0
 \ar@{->}[r]
  &F_1
    \ar@{->}[d]
\ar@{->}[r]
 & F_2
    \ar@{->}[d]
\ar@{->}[r]
& F_3
    \ar@{->}[d]
    \ar@{->}[r]
    &0
 \\
 0
 \ar@{->}[r]
  & M_3
\ar@{->}[r]
\ar@{->}[d]
 & M_2
\ar@{->}[r]
\ar@{->}[d]
& M_1
\ar@{->}[r]
\ar@{->}[d]
&0
\\
\
&0
&0
&0
\
 }
\]
Note $F_2 \cong F_1 \oplus F_3$. So $e_1(F_2) = e_1(F_1) + e_1(F_3)$. However
$e_1(M_2) \geq e_1(M_1) + e_1(M_3)$ and $e_1(N_2) \geq e_1(N_1) + e_1(N_3)$.

By Lemma 2.4 we have $e^T_A(M_i) = e_1(F_i) - e_1(M_i) - e_1(N_i)$ for $i = 1,2,3$.
The result follows.
\end{proof}

\s Let us recall the definition of superficial elements.  Let $N$ be an $A$-module. An element $x \in \m \setminus \m^2$ is said to be $N$-\textit{superficial} if there exists $c > 0$ such that $(\m^{n+1}N \colon x)\cap \m^cN = \m^nN$ for all $n \gg 0$. It is well-known that superficial elements exist when the residue field $k$ of $A$ is infinite.
If depth $N > 0$ then one can prove that a $N$-superficial element $x$ is $N$-regular. Furthermore  $(\m^{n+1}N \colon x) = \m^nN$ for all $n \gg 0$. 

\s \label{H-mod-sup} \emph{Behavior of Hilbert coefficients \wrt \ superficial elements:} Assume $N$ is an $A$-module with $\depth N > 0$ and dimension $r \geq 1$. Let $x$ be  $N$-superficial.
Then by \cite[Corollary 10]{Pu1} we have
\[
e_i(N/xN) = e_i(N) \quad \text{for} \ i = 0,\ldots, r-1.
\]

Our next result shows that $e^T_A(-)$ behaves well mod superficial elements.
\begin{proposition}\label{mod-sup}
Suppose $\dim A \geq 2$ and let $M \in \CMa(A)$. Assume the residue field $k$ is infinite. Let $x$ be $A \oplus M \oplus \Om_A(M)$-superficial. Set $B = A/(x)$ and $N = M/xM$. Then
\[
e^T_B(N) = e^T_A(M).
\]
\end{proposition}
\begin{proof}
Note
\begin{align*}
e^T_A(M) &= e_1(A)\mu(M) -e_1(M) - e_1(\Om_A(M)), \\
          &= e_1(B)\mu(N) - e_1(N) - e_1(\Om_A(M)/x\Om_A(M))
\end{align*}
The result follows from observing that $\Om_A(M)/x\Om_A(M) \cong \Om_B(M/xM)$.
\end{proof}

\s We now abstract some of the essential properties of $e^T_A(-)$. 

We say a function $\xi \colon \ICMa(A) \rt \ZZ$ is a \emph{pre-triangle} 
 function if it satisfies the following properties:
\begin{enumerate}
\item
$\xi([M]) \geq 0$ for all $M \in \CMa(A)$.
\item
$\xi([M]) = 0$ if and only if $M$ is free.
\item
$\xi([M_1 \oplus M_2]) = \xi([M_1]) + \xi([M_2])$ for all $M_1, M_2 \in \CMa(A)$.
\item 
\emph{
(sub-additivity)} If $0\rt M \rt N \rt L \rt 0$ is an exact sequence in $\CMa(A)$ then
$$\xi([N]) \leq \xi([M]) + \xi([L]).$$
\end{enumerate}
We state our basic existence result of pre-triangle functions. 
\begin{theorem}\label{existence-C}
Let $(A,\m)$ be a \CM \ local ring of dimension $d \geq 1$. Then the function
$$ e^T_A([M]) = \lim_{n \rt \infty} \frac{(d-1)!}{n^{d-1}}\ell\left(\Tor^A_1(M, \frac{A}{\m^{n+1}}) \right )  \quad \text{where} \  [M] \in \ICMS(A)$$
is a pre-triangle function on $\ICMa(A)$.
\end{theorem}
\begin{proof}
Properties (1), (2) are satisfied by \ref{basic}. Property (3) is trivially satisfied.
Property (4) is satisfied by Theorem \ref{sub-add}.
\end{proof}

\s \label{inf-pre} If $\xi$ is a pre-triangle function then trivially $k\xi$ is a pre-triangle function for any $k \geq 1$. Perhaps less-obvious is  the following:
\begin{proposition}\label{pr-inf-pre}
Let $\xi$ be a pre-triangle function. Then the function \\ $\xi^{(i)} \colon \ICMa(A) \rt \ZZ$ defined by
$$ \xi^{(i)}([M]) = \xi([\Om^i(M)])$$
is a pre-triangle function for all $i \geq 0$. 
\end{proposition}
\begin{proof}
Note $\xi^{(0)} = \xi$. Also note that for $i \geq 2$ we have 
$$\xi^{(i)} = \left(\xi^{(i-1)}\right)^{(1)}.$$
So it suffices to prove that $\nu = \xi^{(1)}$ is a pre-triangle function. 

It is very easy to prove that $\nu$ satisfies properties (1), (2) and (3) and is left to the reader. We prove that $\nu$ satisfies property (4). 
Let $0 \rt M_1 \rt M_2 \rt M_3 \rt 0$ be a short exact sequence in $\CMa(A)$. Note that we have a short exact sequence
\[
0 \rt \Om(M_1) \rt \Om(M_2)\oplus F \rt \Om(M_3) \rt 0;
\] 
where $F$ is a finitely generated free $A$-module (possibly zero).
Since $\xi$ is a pre-triangle function we have
\[
\xi([\Om(M_2])) = \xi([\Om(M_2)\oplus F]) \leq \xi([\Om(M_1)]) + \xi([\Om(M_3)])
\]
The result follows.
\end{proof}
\begin{remark}
In general $\xi^{(i)}$ will be different from $\xi$. For instance if $\xi = e^T_A(-)$ and if the betti-numbers of $M$ are unbounded then note as $e^T_A(M) \geq e_0(\Omega(M)) \geq \mu(\Omega(M))$, see \cite[Lemma 19]{Pu1}, we get that for $i \gg 0$ we have
$e^T(\Omega^i(M)) > e^T(M)$. So in this case $\xi^{(i)}(M) \neq \xi(M)$.
\end{remark}
The following easy proposition (proof left to the reader) combined with \ref{inf-pre} and \ref{pr-inf-pre} yields yet another  abundant number of  pre-triangle functions.

\begin{proposition}\label{add}
Let $\xi_1, \xi_2$ be two pre-triangle functions. Then $\xi = \xi_1 + \xi_2$ is a pre-triangle function. \qed
\end{proposition}
\section{Triangle functions on $\CMS(A)$}
In this section $(A,\m)$ is a Gorenstein local ring of dimension $d \geq 1$ with residue field $k$. Let $\CMa(A)$ denote the full subcategory of maximal \CM \ $A$-modules and let $\CMS(A)$ denote the stable category of maximal \CM \ $A$-modules.
Let $\ICMS(A)$ denote the set of isomorphism classes  in $\CMS(A)$ and for an object $M \in \CMS(A)$ we denote its isomorphism class by $[M]$. In this section we prove Theorem  \ref{existence}. We also construct a large class of triangle functions on $\ICMS(A)$.

\s Let $M \in \CMa(A)$. By $M^*$ we mean the dual of $M$, i.e., $M^* = \Hom_A(M,A)$. Note $M \cong M^{**}$. \\
By $\Om^{-1}(M)$ we mean the \textit{co-syzygy} of $M$. Recall this is constructed as follows.
Let $F \rt G \xrightarrow{\epsilon} M^* \rt 0$ be a minimal presentation of $M^*$. Dualizing we get an exact sequence $0 \rt M \xrightarrow{\epsilon^*}  G^* \rt F^*$. 
Then $\Om^{-1}(M) = \coker \epsilon^*$. It can be easily shown that if $F^\prime \rt G^\prime \xrightarrow{\eta} M^* \rt 0$ is another  minimal presentation of $M^*$ then $\coker \epsilon^* \cong \coker \eta^*$.

\s\emph{Triangulated category structure on $\CMS(A)$}. \\
The reference for this topic is \cite[4.7]{Bu}. We first describe a 
basic exact triangle. Let $f \colon M \rt N$ be a morphism in $\CMa(A)$. Note we have an exact sequence $0 \rt M \xrightarrow{i} Q \rt \Om^{-1}(M) \rt 0$, with $Q$-free. Let $C(f)$ be the pushout of $f$ and $i$. Thus we have a commutative diagram with exact rows
\[
  \xymatrix
{
 0
 \ar@{->}[r]
  & M
\ar@{->}[r]^{i}
\ar@{->}[d]^{f}
 & Q
\ar@{->}[r]^{p}
\ar@{->}[d]
& \Om^{-1}(M)
\ar@{->}[r]
\ar@{->}[d]^{j}
&0
\\
 0
 \ar@{->}[r]
  &N    
\ar@{->}[r]^{i^\prime}
 & C(f)
\ar@{->}[r]^{p^\prime}
& \Om^{-1}(M)
    \ar@{->}[r]
    &0 
\
 }
\]
Here $j$ is the identity map on $\Om^{-1}(M)$.
As $N, \Om^{-1}(M) \in \CMa(A)$  it follows that $C(f) \in \CMa(A)$.
Then the projection of the sequence 
$$ M\xrightarrow{f} N \xrightarrow{i^\prime} C(f) \xrightarrow{-p^\prime} \Omega^{-1}(M)$$
in $\CMS(A)$  is a basic exact triangle. Exact triangles in $\CMS(A)$ are triangles isomorphic to a basic exact triangle.

\begin{remark}\label{ex-tr}
If $0 \rt M \xrightarrow{f} N \rightarrow L \rt 0$ is an exact sequence in $\CMa(A)$ then we have an exact triangle $M \rt N \rt L \rt \Om^{-1}(M)$ in $\CMS(A)$. To see this we do the basic construction with the map $f$. Then note that we have an exact sequence in $\CMa(A)$
\[
0 \rt Q \rt C(f) \rt L \rt 0.
\]
As $A$ is Gorenstein and $Q$ is free we get $C(f) \cong Q \oplus L$. It follows that 
$C(f) \cong L$ in $\CMS(A)$. The result follows.
\end{remark}
 The main result of this section is
 \begin{theorem}\label{pre}
 Let $\xi \colon \ICMa(A) \rt \ZZ$ be a pre-triangle function. Then $\xi$ induces a triangle function $\xi^\prime \colon \ICMS(A) \rt \ZZ$ defined as
 \[
 \xi^\prime([M]) = \xi(<M>).
 \]
 (Here by $<M>$ we mean isomorphism class of $M$ in $\CMa(A)$).
 \end{theorem}
 \begin{proof}
 We first show that $\xi^\prime $ is a well-defined function. Let $[M] = [N]$. Then there exists  free modules $F,G$ such that $M\oplus F \cong N \oplus G$. So $<M\oplus F> = < N \oplus G>$ in $\ICMa(A)$. Thus
 $\xi(<M\oplus F>) = \xi(< N \oplus G>)$. But $\xi$ is a pre-triangle function. So
 $$\xi(<M\oplus F>) = \xi(<M>) + \xi(<F>) = \xi(<M>).$$
 Similarly  $\xi(< N \oplus G>) = \xi(<N>)$. It follows that $\xi^\prime$ is a well-defined function. 
 
 Properties (1),(2),(3) are trivial to show and is left to the reader. We prove property (4). Let $M \rt N \rt L \rt \Om^{-1}(M)$ be an exact triangle in $\CMS(A)$. Then it is isomorphic to a basic triangle $M^\prime \xrightarrow{f} N^\prime \rt C(f) \rt \Om^{-1}(M)$. We have an exact sequence $0 \rt N^\prime \rt C(f) \rt \Om^{-1}(M^\prime) \rt 0$.
 As $\xi$ is a pre-triangle we have
 $$ \xi(<C(f)>) \  \leq \ \xi(<N^\prime>) + \xi(<\Om^{-1}(M^\prime)>).$$
 Note $C(f) \cong L$, $\Om^{-1}M \cong \Om^{-1}(M^\prime)$ and $N \cong N^\prime$ in $\CMS(A)$.  So we have
 $$ \xi^\prime([L]) \leq \xi^\prime([N]) + \xi^\prime([\Om^{-1}(M)]).$$
 Thus we have shown property 4(b) for all exact triangles. By \ref{rotation} it follows that property 4(a),(c) are also satisfied for all exact triangles. 
 \end{proof}
 
 We now give
 \begin{proof}[Proof of Theorem \ref{existence}]
 This follows from Theorem \ref{existence-C} and Theorem \ref{pre}.
 \end{proof}

\s\label{inf-tri} We now give construction of infinitely many triangle functions on $\CMS(A)$. Since we have one pre-triangle function on $\ICMa(A)$, we constructed in \ref{inf-pre},  \ref{pr-inf-pre} and \ref{add} infinitely many pre-triangle functions. Each of these will yield a triangle function on $\CMS(A)$.  

\section{Some preliminaries on Liason of Modules and \\ Maximal \CM \ approximation}
In this section we recall the definition of linkage of modules as given in \cite{MS}.  We also recall the notion of maximal \CM \ approximations and then breifly explain its connection with Liason theory. We also prove an easy result regarding maximal \CM \ approximations (which we suspect is already known but we are unable to find a reference). Throughout this section $A$ is a Gorenstein ring. Recall a Gorenstein ideal $\q$ in $A$ is a perfect ideal $\q$ with $A/\q$ a Gorenstein ring.

\s Let us recall the definition of transpose of a module. Let $F_1 \xrightarrow{\phi} F_0 \rt M \rt 0$ be a minimal presentation of $M$.  Let $(-)^* = \Hom(-,A)$. The \textit{transpose} $\Tr(M)$ is defined by the exact sequence
\[
0 \rt M^* \rt F_0^* \xrightarrow{\phi^*} F_1^* \rt \Tr(M) \rt 0.
\] 

\begin{definition}
Two $A$-modules $M$ and $N$ are said to be \textit{horizontally linked} if 
$M \cong \Om(\Tr(N))$ and $N \cong \Om(\Tr(M))$.
\end{definition}
Next we define linkage in general.
\begin{definition}
Two $A$-modules $M$ and $N$ are said to be linked via a Gorenstein ideal $\q$ if
\begin{enumerate}
\item
$\q \subseteq \ann M \cap \ann N$, and
\item
$M$ and $N$ are horizontally linked as $A/\q$-modules.
\end{enumerate}
We write it as $M \sim_\q N$.
\end{definition}

\begin{remark}
 It can be shown that ideals $I$ and $J$ are linked by a Gorenstein ideal $\q$ (definition as in the introduction) if and only if the  module $A/I$ is linked to $A/J$ by $\q$, see \cite[Proposition 1]{MS}. 
\end{remark}

\s We say $M, N$ are in  \emph{same linkage class} of modules  if there is a sequence of $A$-modules $M_0,\ldots, M_n$  and Gorenstein ideals $\q_0 \ldots,\q_{n-1}$ such that
\begin{enumerate}[\rm (i)]
\item
$M_j \sim_{\q_j} M_{j+1}$, for $j = 0,\ldots, n-1$.
\item
 $M_0 = M$ and $M_n =N$. 
\end{enumerate}
If $n$ is even then we say that $M$ and $N$ are \emph{evenly linked}.

\s\textit{(MCM-approximations)} An MCM approximation of a $A$-module $M$ is a short exact sequence
$ 0 \rt Y \rt X \rt M \rt 0$ where $X$ is maximal \CM \ and $\projdim Y < \infty$. If
$0 \rt Y^\prime \rt X^\prime \rt M \rt 0$ is another MCM approximation of $M$ then $X$ and $X^\prime$ are stably isomorphic, i.e., there exists  free modules $F,G$ with $X\oplus F \cong X^\prime \oplus G$. Thus we have a well-defined object $X_M$ in $\CMS(A)$. This in fact defines a functor but we do not need it here.

The relation between Liason theory and MCM approximation is the following result by
Martsinkovsky and Strooker \cite[Theorem 13]{MS}. For Cohen-Macaulay modules  of codimension $r >0$ this result was proved by Yoshino and Isogawa \cite[Corollary 1.6]{YI}.

\begin{theorem}
Let $(A,\m)$ be a Gorenstein local ring and let $M$ and $N$ be two  $A$-modules. If $M$ is evenly linked to $N$ then $X_M \cong X_N$ in $\CMS(A)$.
\end{theorem}
 
\s \label{codim-n}  If $M$ is \CM \ then maximal \CM \ approximation of $M$ are very easy to construct. We recall this construction from \cite[p.\ 7]{AB}.  Let $n = \codim M = \dim A - \dim M$.  Let $M^\vee = \Ext^n_A(M,A)$. It is well-known that $M^\vee$ is \CM \ module of codim $n$ and $M^{\vee \vee} \cong M$.
Let $\FF$ be any free resolution of $M^\vee$ with each $\FF_i$ a finitely
generated free module. Note $\FF$ need not be minimal free resolution of $M$. Set  $S_n(\FF) = \image( \FF_n \xrightarrow{\partial_n} \FF_{n-1})$. Then  note $S_n(\FF)$ is a maximal \CM \ $A$-module. It can be easily proved that $X_M \cong S_n(\FF)^*$ in $\CMS(A)$.

The following result should be well-known to the experts. We give a proof due to lack of a reference.
\begin{proposition}\label{exact-n}
Let $M,N,L$ be \CM \ $A$-modules with $\codim = n$.  Suppose we have an exact sequence $0 \rt M \rt N \rt L \rt 0$. Then we have an exact triangle
\[
X_M \rt X_N \rt X_L \rt  \Om^{-1}(X_M)
\]
in $\CMS(A)$.
\end{proposition}
\begin{proof}
Dualizing
 we have an exact sequence $0 \rt L^\vee \rt N^\vee \rt M^\vee \rt 0$.
By a well-known theorem in homological algebra there exists a short-exact sequence of complexes $0 \rt \FF \rt \GG \rt \HH \rt 0$ where 
$\FF, \GG$  and $ \HH$ are free resolutions of $L^\vee, N^\vee$ and $M^\vee$ respectively. We use notation as in \ref{codim-n}. Note we have an exact sequence
\[
0 \rt S_n(\FF) \rt S_n(\GG) \rt S_n(\HH) \rt 0.
\]
As each of the modules in the above short exact sequence is maximal \CM \ we get an exact sequence 
\[
0 \rt S_n(\HH)^* \rt S_n(\GG)^* \rt S_n(\FF)^* \rt 0.
\]
The result now follows from \ref{ex-tr} and \ref{codim-n}.
\end{proof}

\section{Proof of Theorem \ref{result-1}}
In this section $(A,\m)$ is a Gorenstein local ring.  First we prove that for one dimensional rings the set of even liason classes of $\{\m^n \mid n \geq 1 \}$ is a finite set. 

\begin{proposition}\label{1-dim}
Let $(A,\m)$ be a one-dimensional Gorenstein ring.  Then there exists $s \geq 1$ such that $\m^{sn + r}$ is evenly linked to $\m^{s(n-1) + r}$ for all $n \gg 0$ and $r = 0,1,\ldots,s-1$.
\end{proposition}
\begin{proof}
Let $a \in \m^s \setminus \m^{s+1}$ be such that image of $a$ in $\m^s/\m^{s+1}$ is a parameter for the associated graded ring $G  = \bigoplus_{n\geq 0} \m^n/\m^{n+1}$. Then it can be shown that $a$ is a non-zero divisor of  $A$ and $(\m^{n+s} \colon a) = \m^n$ for all $n \gg 0$. We also have that $\m^{n+s} = a\m^{n}$ for all $n \gg 0$.

It is easily verified that for all $n \gg 0$ we have $(a^n \colon \m^{sn-r}) = (a^{n-1} \colon \m^{s(n-1) - r}) $ for $r = 0,1,\ldots,s-1$. Therefore $\m^{sn -r}$ is evenly linked to $\m^{s(n-1) -r}$ for $r = 0,1,\ldots, s-1$ and for all $n \gg 0$. 
\end{proof}
\begin{remark}
If the residue field of $A$ is infinite then note we can choose 
$s = 1$ in the above Proposition \cite[1.5.12]{BH}. So we get $\m^n$ is evenly linked to $\m^{n-1}$ for all $n \gg 0$.
\end{remark}
\s By  \cite[Theorem 3.6]{KM}  it follows that  if $K$ is a field 
and $R = K[[X_1,\ldots,X_n]]$ then $\n^i $ is evenly linked to $\n^{i-1}$ for all $i \geq 2$; here $\n$ is the maximal ideal of $R$. We do not know whether in general for a regular local ring $(R,\n)$ with $\dim R \geq 3$ we have $\n^{i}$ is evenly linked to $\n^{i-1}$. We also 
do not know whether the set of even liason classes of $\{\n^i \mid i \geq 1 \}$ is a finite set. 

\s Let $M$ be an   $A$-module of dimension $r$. The function 
$$H(M,n) = \ell(\m^n M/\m^{n+1} M) \quad n  \geq 0,$$ is called the \emph{Hilbert function} of $M$. It is well-known that it is of polynomial type of degree $r -1$. In particular if $r \geq 2$ then $H(M,n) \rt \infty $ as $n\rt \infty$.

We now give:
\begin{proof}[Proof of Theorem \ref{result-1}]
For $n \geq 0$ we have an exact sequence of finite length $A$-modules
\[
0 \rt \frac{\m^nM}{\m^{n+1}M} \rt \frac{M}{\m^{n+1}M} \rt \frac{M}{\m^{n}M}  \rt 0.
\]
For $n \geq 0$,
let $X_n, Y_n$ denote the maximal \CM \ approximations of $\m^nM/\m^{n+1}M$ and $M/\m^{n+1}M$ respectively. Note $X_n \cong X_k^{H(M,n)}$ in $\CMS(A)$.
By \ref{exact-n}, for all $n \geq 1$ we have an exact triangle in $\CMS(A)$
\begin{equation}\label{ex-n}
X_n \rt Y_n \rt Y_{n-1} \rt \Omega^{-1}(X_n).
\end{equation}

Suppose if possible $\Lambda_M \subseteq \bigcup_{i=1}^{m} L_i$ for some finitely many even liason classes $L_1,\ldots, L_n$. Choose $V_i \in L_i$ for $i = 1,\ldots,m$. Then
for all $n \geq 0$ we have $Y_n \cong X_{V_i}$ in $\CMS(A)$ for some $i$ (depending on $n$). Notice we also have 
$\Om^{-1}(Y_n) \cong \Om^{-1}(X_{V_i})$ in $\CMS(A)$.

Let $\xi$ be any triangle function on $\ICMS(A)$. Then by \ref{ex-n} we have
\begin{equation}\label{ex-xi-n}
\xi([\Om^{-1}(X_n)]) \leq \xi([Y_{n-1}]) + \xi([\Om^{-1}(Y_n)]).
\end{equation}
Let 
\begin{align*}
\alpha &= \max \{ \xi([X_{V_i}]) \mid i = 1,\ldots, m \}, \\
\beta &= \max \{ \xi([\Om^{-1}(X_{V_i})]) \mid i = 1,\ldots, m \}. 
\end{align*}
Also note that 
$$\Om^{-1}(X_n) = (\Om^{-1}X_k)^{H(M,n)} \quad \text{in} \ \CMS(A).$$
By \ref{ex-xi-n} we have
\[
H(M,n)\xi([\Om^{-1}X_k]) \leq \alpha + \beta.
\]
Since $\dim M \geq 2$ we have that $H(M,n) \rt \infty $ as $n \rt \infty$. It follows that $\xi([\Om^{-1}X_k]) = 0$. Therefore $\Om^{-1}(X_k)$ is free. It follows that $X_k$ is free. Therefore $\projdim k < \infty$. This implies that $A$ is regular.
\end{proof}

\section{Proof of Theorem \ref{result-2}}
The following result  follows easily  from \cite[Theorem 3.6]{KM}. However we give a proof as we do not have a reference. It also explains the significance of Theorem \ref{result-2}.
\begin{proposition}\label{reg-equi}
Let $(A,\m)$ be a complete equi-characteristic Gorenstein local ring.  Let $I$ be an ideal generated by a regular sequence. The $I^n$ is evenly linked to $I^{n-1}$ for all $n \geq 2$.
\end{proposition}
To prove this result we need the following general result.
\begin{lemma}\label{flat}
Let $\phi \colon (A,\m) \rt (B,\n)$ be a faithfully flat homomorphism of Gorenstein local rings. 
Let $I,J$ be ideals in $A$ and let $\q$ be a Gorenstein ideal in $A$ such that $I \sim_\q J$. Then
\begin{enumerate}[\rm (1)]
\item
$\q B$ is a Gorenstein ideal in $B$.
\item
$IB \sim_{\q B} JB$.
\end{enumerate} 
\end{lemma}
\begin{proof}
(1) Tensoring a minimal free $A$-resolution of $A/\q$ with we get a minimal free $B$-resolution of $B/\q B$. Thus $\projdim_B B/\q B$ is finite.

Let $S = B/\m B$ be the fiber ring of $\phi$. Then as $B$ is Gorenstein we have that $S$ is a Gorenstein ring as well, see \cite[Theorem 23.4]{Mat}. 

Note the induced map $\ov{\phi} \colon A/\q \rt B/\q B$ is also flat with fiber ring $S$. As $A/\q$ and $S$ are Gorenstein rings we have that $B/\q B$ is also Gorenstein, see \cite[Theorem 23.4]{Mat}.  Thus $\q B$ is a Gorenstein ideal.

(2) As $I \sim_\q J$ we have $(\q \colon I) = J$ and $(\q \colon J) = I$. As $\phi$ is flat we have 
\[
JB = (\q \colon I)B = (\q B \colon I B);
\]
see \cite[Theorem 7.4]{Mat}. Similarly $IB = (\q B \colon J B)$. Therefore $IB \sim_{\q B} JB$.
\end{proof}
As an easy consequence we have
\begin{corollary}\label{reg-reg}
Let $K$ be a field. Let $R = K[[X_1,\ldots,X_n]]$. Fix $r \geq 1$. Set $I = (X_1,\ldots,X_r)$. Then $I^n$ is evenly linked to $I^{n-1}$ for $n \geq 2$.
\end{corollary}
\begin{proof}
Let $T = K[[X_1,\ldots,X_r]]$ and let $\m = (X_1,\ldots,X_r)$. The inclusion $T \rt R$ is flat. By \cite[Theorem 3.6]{KM},  $\m^n$ is evenly linked to $\m^{n-1}$ for $n \geq 2$. By \ref{flat} we have that  $I^n$ is evenly linked to $I^{n-1}$ for $n \geq 2$.
\end{proof}
We now give
\begin{proof}[Proof of Proposition \ref{reg-equi}]
Let $I = (x_1,\ldots,x_r)$. Extend this regular sequence to a system of parameters
$x_1,\ldots,x_d$ of $A$. Assume $A = K[[Y_1,\ldots,Y_m]]/I$. Consider the subring
$B = K[[x_1,\ldots,x_d]]$ of $A$. Then note that
\begin{enumerate}
\item
$A$ is finitely generated as a $B$-module.
\item
$B \cong K[[X_1,\ldots,X_d]]$ the power series ring over $K$ in $d$-variables.
\item
As $A$ is \CM \ we have that $A$ is free as a $B$-module. Thus the inclusion $i \colon A \rt B$ is flat.
\end{enumerate}
By Corollary \ref{reg-reg} we have that the $B$-ideal $J = (x_1,\ldots,x_r)$ has the property that $J^n$ is evenly linked to $J^{n-1}$ for all $n \geq 2$. By \ref{flat} it follows that $I^n$ is evenly linked to $I^{n-1}$ for all $n \geq 2$.
\end{proof}
\begin{remark}
(with hypotheses as in \ref{reg-equi}). Note that as modules, $A/I^n$ is evenly linked to $A/I^{n-1}$ for all $n \geq 2$. It follows that if $F$ is a finitely generated free $A$-module then $F/I^nF$ is evenly linked to $F/I^{n-1}F$ for all $n \geq 2$.
\end{remark}
We now give
\begin{proof}[Proof of Theorem \ref{result-2}]
As $M$ is a maximal \CM \ $A$-module it follows that $x_1,\ldots,x_r$ is an $M$-regular sequence.
Note that $I^nM/I^{n+1}M \cong (M/IM)^{\gamma_n}$ where $\gamma_n = \binom{n+r-1}{r-1}$, see \cite[Theorem 1.1.8]{BH}.
For all $n \geq 0$ we also have an exact sequence 
\begin{equation}\label{r2-eq1}
0 \rt \frac{I^nM}{I^{n+1}M} \rt \frac{M}{I^{n+1}M} \rt \frac{M}{I^{n}M} \rt 0.
\end{equation}
Inductively one can prove that $M/I^nM$ is a \CM \ $A$-module of codimension $r$.
Thus \ref{r2-eq1} is an exact sequence of codimension $r$ \CM \ $A$-modules.
For $n \geq 0$ let $X_n, Y_n$ denote maximal \CM \ approximations of $I^nM/I^{n+1}M$ and $M/I^{n+1}M$ respectively. Therefore by \ref{exact-n} for all $n \geq 1$ we have the following exact triangle in $\CMS(A)$
\begin{equation}\label{r2-e2}
X_n \rt Y_n \rt Y_{n-1} \rt \Om^{-1}(X_n).
\end{equation} 
Suppose if possible $\Lambda_M^I \subseteq \bigcup_{i=1}^{m} L_i$ for some finitely many even liason classes $L_1,\ldots, L_n$. Choose $V_i \in L_i$ for $i = 1,\ldots,m$. Then
for all $n \geq 0$ we have $Y_n \cong X_{V_i}$ in $\CMS(A)$ for some $i$ (depending on $n$). Notice we also have 
$\Om^{-1}(Y_n) \cong \Om^{-1}(X_{V_i})$ in $\CMS(A)$.

Let $\xi$ be any triangle function on $\ICMS(A)$. Then by \ref{ex-n} we have
\begin{equation}\label{ex-xi-n2}
\xi([\Om^{-1}(X_n)]) \leq \xi([Y_{n-1}]) + \xi([\Om^{-1}(Y_n)]).
\end{equation}
Let 
\begin{align*}
\alpha &= \max \{ \xi([X_{V_i}]) \mid i = 1,\ldots, m \}, \\
\beta &= \max \{ \xi([\Om^{-1}(X_{V_i})]) \mid i = 1,\ldots, m \}. 
\end{align*}
Also note that 
$$\Om^{-1}(X_n) = (\Om^{-1}X_{M/IM})^{\gamma_n} \quad \text{in} \ \CMS(A).$$
By \ref{ex-xi-n2} we have
\[
\gamma_n\xi([\Om^{-1}X_{M/IM}]) \leq \alpha + \beta.
\]
Since $r \geq 2$ we have that $\gamma_n \rt \infty $ as $n \rt \infty$. It follows that $\xi([\Om^{-1}X_{M/IM}]) = 0$. Therefore $\Om^{-1}(X_{M/IM})$ is free. It follows that $X_{M/IM}$ is free. Therefore $\projdim_A M/IM < \infty$. As $x_1,\ldots,x_r$ is an $M$-regular sequence it follows that $\projdim_A M$ is finite.  So $M$ is free.
\end{proof}

\section{Proof of Theorem \ref{result-3}}
Let $r \geq 1$. Let $\CMr(A)$ denote the full sub-category of \CM \ $A$-modules of codimension $r$. In this section we define an invariant of modules in $\CMr(A)$ and then use it to prove Theorem \ref{result-3}. 

\begin{definition}
Let $N \in \CMr(A)$. Let $X_N$ be a maximal \CM \ approximation of $N$. Set 
$\theta_A(N) = e^T_A([X_N])$. 
\end{definition}
As $e^T_A(-)$ is a triangle function on $\CMS(A)$ it follows that $\theta_A(N)$ is a well-defined invariant of $M$.

The number $\theta_A(-)$ behaves well mod superficial sequences. Let us recall the notion of a superficial sequence. Let $N$ be an $A$-module of dimension $r$.
 We say $\bx = x_1,\ldots,x_s$ (with $s \leq r$) is an $N$-superficial sequence if $x_1$ is $N$-superficial, $x_i$ is $N/(x_1,\ldots,x_{i-1})N$ superficial for $2 \leq i \leq s$.
\begin{proposition}\label{estimate}
Let $N \in \CMr(A)$ with $r \geq 1$ and let $\dim A = d$. Let $0 \rt Y \rt X \rt N \rt 0$ be a maximal \CM \ approximation of $M$. Let $\bx = x_1,\ldots,x_{d-r}$ be a $A\oplus X \oplus \Omega(X) \oplus N$ superficial sequence. Set $B = A/(\bx)$. Then
$$ \theta_A(N) \leq e_0(N)\theta_B(k). $$
\end{proposition}
\begin{proof}
Note $\bx$ is a $X\oplus Y \oplus N$ regular sequence. So $0 \rt Y/\bx Y \rt X/\bx X \rt N/\bx N \rt 0$ is a maximal \CM \ approximation of the $B$-module $N/\bx N$. Note as $r \geq 1$ we have that $d-r \leq  = d-1$. Using
\ref{mod-sup} we get $e^T_B(X/\bx X) = e^T_A(X)$. So we have $\theta_A(N) = \theta_B(N/\bx N)$. It suffices to prove   $\theta_B(N/\bx N) \leq e(N)\theta_B(k)$. By \ref{H-mod-sup} we get that $N/\bx N$ is a $B$-module of length $e_0(N)$. 

Let $L$ be a finite length $B$-module. We prove by induction on $\ell(L)$ that $\theta_B(L) \leq \ell(L) \theta_B(k)$. We have nothing to prove if $\ell(L) = 1$. So assume $\ell(L)= m \geq 2$ and the result is proved for all $B$-modules of length $\leq m-1$. 

We have an exact sequence $0 \rt V \rt L \rt k \rt 0$ where $\ell(V) = m-1$. In $\CMS(B)$ we have an exact triangle
\[
X_V \rt X_L \rt X_k \rt \Om^{-1}(X_V)
\]
It follows that $e^T_B(X_L) \leq e^T_B(X_V) + e^T_B(X_k) \leq m e^T_B(k)$. Thus $\theta_B(L) \leq m \theta_B(k)$.
\end{proof}
Our proof of Theorem \ref{result-3} uses the following result by Herzog-Kuhl, \cite[Theorem 2.1]{HK}.
\begin{theorem}
Let $R$ be a local Gorenstein domain with infinite residue field $k$. Let $0 \rt F_1 \rt M_1 \rt I_1 \rt 0$ and 
 $0 \rt F_2 \rt M_2 \rt I_2 \rt 0$ be any two Bourbaki sequences (i.e., $F_1,F_2$ are free, $M_1, M_2$ are maximal \CM \ modules and $I_1, I_2$ are \CM \ ideals of codimension $2$). Then the following two statements are equivalent:
 \begin{enumerate}[\rm (1)]
 \item
 $M_1$ and $M_2$ are stably isomorphic.
 \item
 $I_1$ and $I_2$ are evenly linked by a complete intersection.
 \end{enumerate}
\end{theorem}
We should remark that a Bourbaki sequence is simply a maximal \CM \ approximation of $I$ where $I$ is a codimension $2$ \CM \ ideal.
\s Our proof of Theorem \ref{result-3} also uses the fact that if $A$ is a simple hypersurface singularity then it has only finitely many indecomposable maximal \CM \ modules; \cite{K}. We now give
\begin{proof}[Proof of Theorem \ref{result-3}]
Suppose if possible for some $m \geq 1$ the set $\mathcal{C}_m $ is not contained in any collection of finitely many even liason classes. For $j \geq 1$ let $I_j$ be ideals with $e_0(A/I_j) \leq m$ such that the liason classes $L_j$ of $I_j$ are all distinct. 

For $j \geq 1$ let $0 \rt Y_j \rt X_j \rt A/I_j \rt 0$ be maximal \CM \ approximation of $A/I_j$.

\noindent \emph{Claim:}
 there exists $\bx = x_1,\ldots,x_{d-2}$ such that $\bx$ is a $A \oplus X_j \oplus \Om(X_j) \oplus A/I_j$-superficial for all $j \geq 1$.
 
 To prove the claim let us recall a construction of a superficial element. Let $E$ be an $A$-module. Consider the associated graded ring $G = \bigoplus_{n \geq 0} \m^n/\m^{n+1}$ of $A$ and let $G(E) = \bigoplus_{n\geq 0} \m^n E/\m^{n+1}E$ be the associated graded module of $E$. Let $^*\Ass(G(E))$ be the relevant associated primes of $G(E)$, i.e., those associated primes of $G(E)$ which are not equal to $G_+ =  \bigoplus_{n \geq 1} \m^n/\m^{n+1}$. Let $V = \m/\m^2$. Note $P\cap V$ is a proper subspace of $V$ for every $P \in \ ^*\Ass(G(E))$. As the field $k$ is infinite and as $^*\Ass(G(E))$ is a finite set we get that
 \[
\widetilde{V} = \left( V \setminus \bigcup_{P \in ^*\Ass(G(E))} P \cap V \right) \neq \emptyset.
 \]
An element $x \in \m$ such that $\ov{x} \in \widetilde{V}$ is an $E$-superficial element. 
This construction yields superficial element of a single module. Note we have only used that $k$ is an infinite field. 

To prove the claim, let $E_j = A \oplus X_j \oplus \Om(X_j) \oplus A/I_j$ for $j \geq 1$. Note $^*\Ass(G(E_j))$
is a finite set. So 
\[
D = \bigcup_{i\geq 1}\left(\bigcup_{P \in ^*\Ass(G(E_i))} P \cap V\right)
\]
is a union of a countable number of \textit{proper} $k$-subspaces of $V$. As $k$ is an uncountable field we get that 
$\widetilde{V} =  V \setminus D$ is non-empty. An element $x \in \m$ such that $\ov{x} \in \widetilde{V}$ is an $E_j$-superficial element for all $j \geq 1$. Note for construction of this superficial element we did not use that $A$ is a simple hypersurface singularity. We only used that the residue field of $A$ is uncountable. Iterating this procedure we get a superficial sequence $\bx = x_1,\ldots,x_{d-2}$ for all $E_j$. 

Set $B = A/(\bx)$. By the claim and Proposition \ref{estimate} we get that 
$$ e^T_A(X_j) = \theta_A(A/I_j) \leq m \theta_B(k).$$
  Set $ c = m \theta_B(k).$ Let $M_1,M_2, \ldots,M_m$ be all the indecomposable non-free maximal \CM \ $A$-modules.
  Write
  $$ X_j = M_1^{a_{1,j}} \oplus \cdots\oplus M_m^{a_{m,j}}\oplus A^{l_j}.$$
  Here $a_{i,j}, l_j \geq 0$. Note that
  \[
  \sum_{j = 1}^{m}a_{i,j} \leq  e^T_A(X_j)  \leq c \quad \text{for all} \ j \geq 1.
  \]
  By pigeon-hole principle it follows that there exists $r,s$ with $r < s$ such that $X_r$ is stably isomorphic to $X_s$. Note $\Om(X_r)$ (and a free summand) will give a maximal \CM \ approximation of $I_r$. By Herzog-Kuhl's result we have that $I_r$ is evenly linked to $I_s$. So $L_r = L_s$ a contradiction. Thus $\mathcal{C}_m$ is contained in finitely many even liason classes of $A$.
\end{proof}

\end{document}